\documentclass[12pt,flegn]{amsproc}

\usepackage[margin=1.2in]{geometry}  
\usepackage{graphicx}              
\usepackage{amsmath}               
\usepackage{amsfonts}              
\usepackage{amsthm}                
\usepackage{amssymb}
\usepackage{appendix}
\usepackage{longtable}
\usepackage{todonotes}  
\usepackage{enumerate}
\usepackage{mathtools}
\usepackage{floatrow}

\newtheorem{thm}{Theorem}[section]
\newtheorem{theorem}[thm]{Theorem}

\newtheorem{prop}[thm]{Proposition}
\newtheorem{cor}[thm]{Corollary}
\newtheorem{conj}[thm]{Conjecture}

\numberwithin{equation}{section}
\theoremstyle{remark}

\newcommand{\mat}[1]{\left(\begin{matrix} #1 \end{matrix} \right)}  

\newcommand{\la}{\langle}
\newcommand{\ra}{\rangle}
\newcommand{\ccr}{}

\begin{document}
\raggedbottom
\setlength\parindent{24pt}

\title{Logarithmic Diameter Bounds for Some Cayley Graphs}
\author{Lam Pham}
\address{Department of Mathematics, Brandeis University, MA, USA}
\email{lampham@brandeis.edu}

\author{Xin Zhang}
\address{Department of Mathematics, The University of Hong Kong, Hong Kong}
\email{xzhang@maths.hku.hk}

\begin{abstract} 
Let $S\subset\text{GL}_n(\mathbb Z)$ be a finite symmetric set. We show that if the Zariski closure of $\Gamma=\langle S\rangle$ is { a product of special linear groups} or {a special affine linear group}, then the diameter of the Cayley graph $\text{Cay}(\Gamma/\Gamma(q),\pi_q(S))$ is $O(\log q)$, where $q$ is an arbitrary positive integer, $\pi_q:\Gamma\to \Gamma/\Gamma(q)$ is the canonical projection induced by the reduction modulo $q$, and the implied constant depends only on $S$.

\end{abstract}

\maketitle

\section{Introduction}

Let $n\in\mathbb N$, and let $\Gamma=\langle S\rangle \subset \text{GL}_n(\mathbb Z)$ be a group generated by a finite symmetric generating set $S$. (This means $S^{-1}=S$.) For any $q\in\mathbb N$, let $\pi_q:\mathbb Z\to \mathbb Z/q\mathbb Z$ be the reduction modulo $q$. It induces a homomorphism $\text{GL}(n,\mathbb Z)\to \text{GL}(n,\mathbb Z/q\mathbb Z)$, which we still denote by  $\pi_q$. Consider the Cayley graphs
$$
X_q:=\text{Cay}(\pi_q(\Gamma),\pi_q(S)),\quad
q\in\mathbb N.
$$
Any probability measure $\mu$ on $X_q$ gives rise to a convolution operator
$$
T_\mu^{(q)}:\ell^2_0(X_q)\to \ell^2_0(X_q),\quad
f\mapsto \mu*f.
$$
Here, for any finite set $X$, we write $\ell^2(X)=\ell^2_0(X)\oplus \mathbb C$, i.e., $\ell^2_0(X_q)$ is the codimension one subspace orthogonal to the constants. Of particular interest is when $\mu=\frac{1}{|S_q|}\sum_{s\in S_q} \delta_s$, the uniform probability measure on $S_q=\pi_q(S)$. In this case, we will simply write $T_S^{(q)}$. 

The group $\Gamma$ is said to have \emph{super-approximation} (or Property $\tau$) with respect to a set $\mathcal C\subset \mathbb N$ if there exists $\epsilon>0$ such that
$$
\|T_S^{(q)}\|\leq 1-\epsilon,\quad \forall q\in \mathcal C.
$$
If $\mathcal C=\mathbb N$, $\Gamma$ is said to have \emph{super-approximation}.

It is well-known that an equivalent statement for $\Gamma$ to have super-approximation with respect to $\mathcal C$ is that the family of Cayley graphs $\{X_q\}_{q\in \mathcal C}$ is an \emph{expander family}, i.e., that there exists $\epsilon>0$ such that
$$
\min \left\{
\frac{|\partial U|}{|U|}
\vert U\subset X_q,~|U|\leq \frac{|X_q|}{2}
\right\}\geq \epsilon,\quad \forall q\in \mathcal C.
$$
Note that even though the constant $\epsilon>0$ depends in general on the choice of generators, the fact that $\Gamma$ has super-approximation does not. 



In \cite{BourgainGamburd2008}, Bourgain and Gamburd introduced a powerful technique, now known as the \emph{Bourgain-Gamburd expansion machine}, to prove super-approximation.  A crucial ingredient of \cite{BourgainGamburd2008} is Helfgott's Product Theroem \cite{Helfgott2008}, which admits later generalizations  \cite{BreuillardGreenTao2011,PyberSzabo2016}.  We record the following state-of-the-art results on super-approximation:

\begin{theorem}[\cite{BourgainVarju2012}]\label{BV2012}
Every Zariski-dense subgroup of $\text{SL}(d,\mathbb Z)$ has super-approximation.
\end{theorem}


\begin{theorem}[\cite{GolsefidyVarju2012,SalehiGolsefidy2019}]\label{GV2012-2019}
Let $\Gamma\subset\text{GL}_n(\mathbb Z)$ be a finitely generated group. Then, $\Gamma$ has super-approximation with respect to:
\begin{enumerate}
\item fixed powers of square-free integers, and
\item powers of primes,
\end{enumerate}
if and only if the connected component of the Zariski closure $\mathbb G$ of $\Gamma$ is perfect, i.e., $[\mathbb G^\circ,\mathbb G^\circ]=\mathbb G^\circ$.
\end{theorem}

Theorem \ref{BV2012} and Theorem \ref{GV2012-2019} naturally lead to the so-called \emph{super-approximation Conjecture}:
\begin{conj}\label{SAC}

Let $\Gamma\subset\text{GL}_n(\mathbb Z)$ be a finitely generated group. Then, $\Gamma$ has super-approximation (without restriction on the moduli) if and only if { the Zariski closure $\mathbb G$ of $\Gamma$ has perfect identity component}.
\end{conj}

The square-free case of Theorem \ref{GV2012-2019} has applications to \emph{affine sieves}:  Given a group $\Gamma<\text{GL}_n(\mathbb Z)$ { whose Zariski closure has perfect identity component}, a vector $v\in \mathbb Z^n$, and an integral polynomial $f$ of $n$ variables, the square-free case of Theorem \ref{GV2012-2019} implies that the set $\{f(\gamma\cdot v): \gamma\in \Gamma\}$ contains infinitely many almost-primes \cite{BGS10, GS13}.  In some special cases of $\Gamma, v, f$, a much stronger statement known as the \emph{local-global principle} holds.  See \cite{BK10, BK14a, BK14, Zh16, Zh18, FSZ19}.  A crucial ingredient of the proof of the local-global principle in these settings is the super-approximation property for the relevant group without any moduli restriction.  We believe a full proof of Conjecture \ref{SAC} would be very interesting and potentially has applications to other arithmetic problems. \par

An easy but important consequence of the super-approximation property is that expander families $(X_q)_{q\in \mathcal C}$ have \emph{logarithmic diameter}. This means that there exists a constant $C(S)>0$ depending only on the generating set $S$ and in particular is independent of $q$, such that
$$
\text{Diam}(X_q)\leq C(S)\log|X_q|,\quad \forall q\in \mathcal C.
$$
Our graphs are Cayley graphs of subgroups of $\text{GL}(d,\mathbb Z/q\mathbb Z)$, so in particular, $|X_q|\leq q^{d^2}$, so every such expander family satisfies the diameter bound
\begin{align}\label{1550}
\text{Diam}(X_q)\leq C(S)\log(q),\quad \forall q\in \mathcal C.
\end{align}

Theorem \ref{BV2012} and Theorem \ref{GV2012-2019} thus lead to

\begin{cor} \label{1611}With notations above, the inequality \eqref{1550} holds if \begin{enumerate}
\item ${\normalfont{\text{Zcl}}} (\Gamma)={{\normalfont\text{SL}}}_d,\mathcal C=N$,
\item ${\normalfont{\text{Zcl}}} (\Gamma)\text{ has perfect identity component,     }  \mathcal C=\{\text{fixed powers of square free numbers}\}$
\end{enumerate}
\end{cor} 

The main result of this paper is to prove logarithmic diameter bounds for new families of groups for which super-approximation is not yet known, to serve as further evidence of the super-approximation for these groups. Let us denote by $\text{SA}_d$ the special affine group, i.e., $\text{SA}_d(R)=\text{SL}_d(R)\ltimes R^d$ for any commutative ring $R$.

\begin{theorem}
\label{main-theorem}
Let $S\subset\text{GL}_n(\mathbb Z)$ be a finite symmetric subset and let $\Gamma=\la S\ra$. If the Zariski closure $\mathbb G$ of $\Gamma$ is $\text{SA}_d$ or $\text{SL}_{d_1}\times \text{SL}_{d_2}\times\cdots\times \text{SL}_{d_k}$, then $\forall q\in \mathbb N$,
\begin{equation}\label{inequality1}
\text{Diam}(X_q)\leq C(S)\log(q),
\end{equation}
where $C(S)>0$ is a constant depending on $S$ but is independent of $q$.
\end{theorem}

In the proof of Theorem \ref{main-theorem} we crucially use {\ccr Corollary \ref{1611} that depends on} Theorem \ref{BV2012} and Theorem \ref{GV2012-2019}.  Another important ingredient is a bounded generation result for $\text{SL}_d$ (Proposition \ref{key-proposition}).  \par

A general perfect group can be written as a semidirect product of a semisimple group and a unipotent group, and a semisimple group can be further decomposed as an almost product of simple groups.  One then naturally wonders if properties of some simple groups can be extended to groups built upon these groups.   While extending the spectral gap property remains a technical challenge for us, we manage to show how to extend the logarithmic diameter property, by illustrating two family of groups $\text{SL}_n^m$ and $\text{SA}_d$.  These two families represent two typical situations of building perfect groups, that is, taking product of simple groups and taking semidirect product of a semisimple group and a unipotent group.  We are only focused on groups whose semisimple parts are (a product of) $\text{SL}_d$, because Theorem \ref{BV2012} is the only known result on super-approximation without moduli restriction.  Very recently,  De Saxc\'e and He \cite{He19} generalises Theorem A of Bourgain-Furman-Mozes-Lindenstrass \cite{BFML}, which is a main ingredient in the proof of Theorem \ref{BV2012}.  It is promising that Theorem \ref{BV2012} can be extended to a general simple group, combined with which our method may be adapted to obtain logarithmic diameter bounds for a general perfect group.  

\par 

\subsection*{Acknowledegements}\

We thank the anonymous referee for the numerous helpful corrections and suggestions on this work.

\section{Proof of the main theorem}

\subsection{Notations}

{ Let $\bf G<\bf GL_n$ and $G=\bf G(\mathbb Z)$.}
The map
$$
\pi_{q}:
\mathbb Z\to \mathbb Z/{q}\mathbb Z,\quad
x\mapsto x\ (\text{mod } {q})
$$
induces a group homomorphpism $\text{GL}(d,\mathbb Z)\to\text{GL}(d,\mathbb Z/q\mathbb Z)$, which we still denote by $\pi_q$. Hence, we obtain homomorphisms
$${\pi_q: G\to {\bf{G}}( \mathbb Z/q \mathbb Z)}
$$
Let $G({{q}})=G\cap \ker(\pi_{{q}})$ denote the congruence subgroup of level ${q}$ of $G(\mathbb Z)$, and let $ G_q=G/G(q)$. For two integers $q_2|q_1$, let { $\varphi_{{{q}}_1,{{q}}_2}$ be the canonical projection map $G_{{{q}}_1}\to G_{{{q}}_2}$}. For a group $\Gamma< G$, we let $\Gamma(q)$=$G(q)\cap \Gamma$. We still denote by $\pi_{{q}}:\Gamma\to\Gamma/\Gamma({q})$ the restriction to $\Gamma$, and the canonical projection map $\Gamma_{{{q}}_1}\to \Gamma_{{{q}}_2}$ by   $\varphi_{{{q}}_1,{{q}}_2}$.  Finally, we write  {${\Gamma_{q}}=\pi_{{q}}(\Gamma)$}.  \par
In $G=\text{SL}(d,\mathbb Z)\ltimes\mathbb Z^d$, we denote by $\theta:G\to \text{SL}_d(\mathbb Z)$ the quotient homomorphism, and we define the (set-theoretic) projection
$$
\tau:G\to \mathbb Z^d,\quad
g\mapsto {g{\vec{0}}},
$$
where $\vec{0}$ is the zero vector of $\mathbb Z^d$.
With this choice, { we may represent $g\in G$ by the pair $(\theta(g), \tau(g))$, and parametrize $G$ as the set $\text{SL}_d(\mathbb Z)\times\mathbb Z^d$ equipped with the product law}
$$
(h_1,u_1)\cdot (h_2,u_2)=(h_1h_2,h_1u_2+u_2),\quad (h_i,u_i)\in \text{SL}_d(\mathbb Z)\times\mathbb Z^d,
$$
where $\text{SL}_d(\mathbb Z)$ acts on $\mathbb Z^d$ by the usual linear action. 

For a prime $p$ and an integer $q$, the notation $p^n||q$ means $p^n|q$ but $p^{n+1}\nmid q$. We also write $\text{ord}_p(q)=n$.

\subsection{{A bounded generation result for $\text{SL}_d(\mathbb Z/q\mathbb Z)$}}
In this section we take $G=\text{SL}_d(\mathbb Z)$. { We take $L\geq 2$ to be a sufficiently large integer.  The exact value of $L$ is to be determined at \eqref{1503}}.\par 

Let us write ${{q}}=\prod_{i=1}^{r_0} p_i^{\alpha_i}$ for the prime decomposition of ${q}$. We assume that $\min_i \alpha_i\geq 4(L-1)$.  

Put
$$
{{q}}_0=\prod_{i=1}^{r_0} p_i,\quad
{{q}}_1=\prod_{i=1}^{r_0} p_i^{L},\quad
{{{q}}_2=\prod_{i=1}^{r_0} p_i^{4(L-1)}},\quad
$$
Our {arguments} rely on the following Proposition.

\begin{prop}\label{key-proposition} {Suppose $L$ is sufficiently large, so that for every $p$, $\mathbb Z/ p^L\mathbb Z$ has at least $d$ invertible elements.  }There exists $N(d)$ such that the following holds. Let $g_0=(a_{jk})_{jk}\in G_q$ be {any} element satisfying the following conditions:
{
\begin{enumerate}
\item All elements in the main diagonal of $\varphi_{q,q_1}(g_0)$ are invertible and distinct in $\mathbb Z/q_1 \mathbb Z$.
\item For any lower triangular entry $a_{jk},j>k$, $p_i^{L-1}||a_{jk}$ for $1\leq i\leq r_0$.  
\item For any {upper} triangular entry $a_{jk},j<k$, {$p_i^{2(L-1)}||a_{jk}$} for $1\leq i\leq r_0$.  
\end{enumerate}}

{\ccr \noindent Similarly, let $g_0'=(a'_{jk})_{jk}\in G_q$ be {any} element satisfying the following conditions:
{
\begin{enumerate}
\item All elements in the main diagonal of $\varphi_{q,q_1}(g_0')$ are invertible and distinct in $\mathbb Z/q_1 \mathbb Z$.
\item For any upper triangular entry $a'_{jk},j<k$, $p_i^{L-1}||a'_{jk}$ for $1\leq i\leq r_0$.  
\item For any lower triangular entry $a'_{jk},j>k$, {$p_i^{2(L-1)}||a'_{jk}$} for $1\leq i\leq r_0$.  
\end{enumerate}} }

\ccr {Then,
$$
\textstyle
{G_q(q_2)\subset\prod_{N(d)} \{gg_0^{\pm 1}g^{-1}, gg_0'^{\pm 1}g^{-1} | g\in G_q(q_0^{L-1})\}.}
$$
In other words, every element $h\in G_q(q_2)$ may be written as a product of $N(d)$ elements of the form 
$gg_0^{\pm}g^{-1}$ or $gg_0'^{\pm}g^{-1}$.}
\end{prop}

\begin{proof} We first prove the proposition in the special case that $q=p^r$ in Step 1- Step 5.  In Step 6 we prove the general case. 
\medskip

{\bf Step 1.} we claim that we can conjugate the matrix $g_0$ to a matrix of the form
$$
{\ccr g_1}=\mat{\lambda_1&0&\cdots&0\\b_{21}&\lambda_2&\cdots &0\\\vdots&\vdots&\ddots&\vdots\\b_{d1}&b_{d2}&\cdots&\lambda_d},\quad
$$
with $\{\lambda_i: 1\leq i\leq d\}$ are distinct $\text{mod } p^{L}$, and $p^{L-1}||b_{jk}$ for $j>k$. 
Indeed, first we want to find an element
$$
x=\mat{1&x_2&\cdots&x_d\\0 &1&\cdots&0\\\vdots&\vdots&\ddots&\vdots\\0&0&\cdots&1}\in {\ccr G_{p^r}({\text{mod }} p^{L-1})}$$
such that
$$
(xAx^{-1})_{1j}\equiv 0\ (\text{mod } p^r),\quad \forall 2\leq j\leq d.
$$
A computation shows that 
$$
xg_1x^{-1}=\mat{a_{11}+x_2a_{21}+\hdots+x_da_{d1}&F_2&\cdots& F_d\\a_{21}& -x_2a_{21}+a_{22}&\cdots& -x_da_{21}+a_{2d}\\\vdots&\vdots&\ddots&\vdots \\a_{d1}&-x_2a_{d1}+a_{d2}&\cdots& -x_da_{d1}+a_{dd}},
$$
where 
\begin{align*}
F_2&=-x_2(x_2a_{21}+\cdots x_da_{d1})+(a_{22}-a_{11})x_2+x_3a_{32}+\hdots+x_da_{d2}+a_{12},\\
&\ \ \vdots\\
F_d&=-x_d(x_2a_{21}+\hdots+x_da_{d1})+x_2a_{2d}+\hdots+x_d(a_{dd}-a_{11})+a_{1d}.
\end{align*}
We want to solve $x_2,\,\hdots,\, x_d$ with {\ccr$x_2\equiv\cdots\equiv x_d\equiv 0({\text{mod }} p^{L-1})$} for the system of equations \begin{equation}\label{system1}
\begin{aligned}
F_2(x_2,\cdots, x_d)&\equiv 0~ (\text{mod } p^r)\\ &\ \ \vdots\\ F_d(x_2,\cdots, x_d)&\equiv 0~ (\text{mod } p^r)\end{aligned}
\end{equation}
 
By the assumption for $g_0$, $a_{11}, a_{22}, \cdots, a_{dd}$ are distinct in $\mathbb Z/p^r\mathbb Z$, so we have {\ccr for all ${2}\leq i\leq d$}, $$\text{ord}_p(a_{ii}-a_{11})\leq L-1.$$ We also have $\text{ord}_p(a_{jk})=L-1, j>k$ and $p^{2(L-1)} | a_{jk}$ for $j<k$.  Therefore, for each $2\leq i\leq d$, we can let $\tilde F_i=F_i/p^{\text{ord}_{p}(a_{ii}-a_{11})}$, so that the coefficient for the linear term $x_i$ for $\tilde F_i$ is coprime to $p$.  To solve  \eqref{system1}, it suffices to solve 
\begin{equation}
\begin{aligned}\label{system2}
\tilde F_2(x_2,\cdots,x_d)&\equiv 0~ (\text{mod } p^s)\\ &\ \ \vdots\\ \tilde F_d(x_2,\cdots, x_d)&\equiv 0~ (\text{mod } p^s)\end{aligned}
\end{equation}
with {\ccr$x_2\equiv \cdots\equiv x_d\equiv 0({\text{mod }} p^{L-1})$ for every $s\geq L-1$. }

We see that if $s=L-1$, the system of equations \eqref{system1} admits a solution $x_2\equiv \hdots\equiv x_d\equiv 0\ (\text{mod } p^{L-1})$.  Moreover, the Jacobian
$$
\left(\frac{ \partial \tilde {F_{i}}}{\partial x_j}\right)_{x_2=\hdots= x_d=0}= \mat{ \frac{a_{22}-a_{11}}{p^{\text{ord}_p(a_{22}-a_{11})}}&\frac{a_{23}}{p^{\text{ord}_p(a_{22}-a_{11})}}&\cdots&\frac{a_{2d}}{p^{\text{ord}_p(a_{22}-a_{11})}}\\\frac{a_{32}}{p^{\text{ord}_p(a_{33}-a_{11})}}&\frac{a_{33}-a_{11}}{p^{\text{ord}_p(a_{33}-a_{11})}}&\cdots &\frac{a_{3_d}}{p^{\text{ord}_p(a_{33}-a_{11})}}\\\vdots&\vdots&\ddots&\vdots\\\frac{a_{d2}}{p^{\text{ord}_p(a_{dd}-a_{11})}}&\frac{a_{d3}}{p^{\text{ord}_p(a_{dd}-a_{11})}}&\cdots&\frac{a_{dd}-a_{11}}{p^{\text{ord}_p(a_{dd}-a_{11})}} },
$$
Reduced mod $p$, the above matrix is lower triangular and invertible in $\text{Mat}_d(\mathbb Z/p\mathbb Z)$.  {\ccr Applying Hensel's Lemma in several variables,  the solution $(0,\hdots,0)$ to \eqref{system2} at level $s=L-1$ can be lifted to an arbitrary level,} which in turn implies the solvability of \eqref{system1}.  

Since $x_2,x_3,\cdots, x_n\equiv 0({\text{mod }} p^{L-1})$, the matrix $xg_1x^{-1}$ satisfies the same conditions as $g_1$ in the Statement of Proposition \ref{key-proposition}.  Using the same method as above, We can conjugate $xg_0x^{-1}$ by a matrix $x'$ of the form 
$$
x'=\mat{1&0&0&\cdots&0\\0 &1&x_3'&\cdots&x_d'\\0&0&1&\ddots&\vdots\\\vdots&\vdots&\ddots&1&0\\ 0&0&\cdots &0&1}\in G_{p^r}(p^{L-1})
$$
to create an element with all upper triangular entries on the first two rows congruent to $0 (\text{mod } p^{r})$.  Iterating the above step for another $(d-2)$ times, we can create a lower triangular matrix $g_1$ as desired.

\medskip

{\bf Step 2.} We show that we can write a general unipotent matrix 
$$
g_2=\mat{1&0&\cdots &0\\c_{21}&1&\cdots&0\\\vdots&\vdots&\ddots&\vdots \\c_{d1}&c_{d2}&\cdots&1},\quad
\text{where}\quad
c_{jk}\in p^{2(L-1)}\mathbb Z/p^r\mathbb Z,\ 1\leq k<j\leq d,
$$
as a product of a conjugate of $g_1$, and a conjugate of $g_1^{-1}$.  

First, for any
$$
g_2'=\begin{pmatrix}
\lambda_1&0&\cdots &0\\c_{21}'&\lambda_2&\cdots&0\\\vdots&\vdots&\ddots&\vdots \\c_{d1}'&c_{d2}'&\cdots&\lambda_d\end{pmatrix}
$$
with $\ccr c_{ij}'\equiv b_{ij}({\text{mod }} p^{2(L-1)}), i\geq j$, 
we can find a lower triangular matrix $y=(y_{ij})_{1\leq i, j\leq d}$ with $1$'s on the diagonal and {\ccr $y_{ij}\equiv 0({\text{mod }} p^{L-1})$ for $1\leq j<i\leq d$}, such that
\begin{equation}
\label{2208}
yg_1 y^{-1}=g_2'\ (\text{mod } p^r).
\end{equation}
Indeed, the matrix equation \eqref{2208} gives $\frac{d(d-1)}{2}$ equations regarding the lower triangular entries of $g_2'$.  We observe that the entry $c'_{j+l,j}$ from the $l^{\text{th}}$ lower subdiagonal is given by an integral polynomial involving entries from the first $l-1$ lower subdiagonals from $g_2'$ and entries from $g_1$.  On the first lower subdiagonal entries, the equation \eqref{2208} leads to
\begin{align*}
(\lambda_1-\lambda_2)y_{21}+b_{21}&=c_{21}' \ (\text{mod } p^r)\\
&~~\vdots\\
(\lambda_{d-1}-\lambda_d)y_{d,d-1}+b_{d,d-1}&=c_{d,d-1}'\ (\text{mod } p^r),
\end{align*}
for which we can find solutions $y_{21},\hdots, y_{d,d-1}\in  p^{L-1}\mathbb Z/p^r\mathbb Z$, because the $\lambda_i$ ($1\leq i\leq d$) are distinct $(\text{mod } p^L)$, and $p^{2(L-1)}$ divides $b_{21}-c'_{21},\hdots,b_{d,d-1}-c'_{d,d-1}$.  Assuming all entries from the first $l-1$ lower subdiagonals of $y$ have been solved, on the $l^{\text{th}}$ diagonal, the equation \eqref{2208} leads to 
$$\ccr
(\lambda_{j}-\lambda_{j+l})y_{j+l,j}+b_{j+l,j}-c'_{j+l,j}+\left(\begin{array}{c}\text{An integral quadratic form of $y_{j+s,j}$ with $s<l$ }\\
\text{from $y$ and lower diagonal entries from $g_1$}\end{array}\right)=0,
$$
{\ccr which also admits solutions for $y_{1+l,1},\hdots, y_{d,d-l}\in p^{L-1}\mathbb Z/p^r\mathbb Z$}. By induction this gives the solvability of \eqref{2208}.

We can then write
$$
\mat{1&0&\cdots &0\\c_{21}&1&\cdots&0\\\vdots&\vdots&\ddots&\vdots \\c_{d1}&c_{d2}&\cdots&1}
=\mat{\lambda_1&0&\cdots &0\\0&\lambda_2&\cdots&0\\\vdots&\vdots&\ddots&\vdots \\0&0&\cdots&\lambda_d}^{-1}\cdot\mat{\lambda_1&0&\cdots &0\\\lambda_2c_{21}&\lambda_2&\cdots&0\\\vdots&\vdots&\ddots&\vdots \\\lambda_dc_{d1}&\lambda_d {\ccr c_{d2}}&\cdots&\lambda_d},
$$
\medskip

{\ccr {\bf Step 3.} Running the same arguments as in the previous two steps for $g_0'$, we can write a general upper triangular matrix 
$g_3=\mat{1&f_{12}&\cdots &f_{1n}\\0&1&\cdots&f_{2n}\\\vdots&\vdots&\ddots&\vdots \\0&0&\cdots&1},\quad
\text{where}\quad
f_{jk}\in p^{2(L-1)}\mathbb Z/p^r\mathbb Z,\ 1\leq j<k\leq d,
$
as a product of {\ccr a bounded number (depending only on $d$) of } conjugates of $g_0'$ and ${g_0'^{-1}}$.
}

{\bf Step 4.} Let $\{\vec{e_i}\}$ be the standard basis on $(\mathbb Z/p^r\mathbb Z)^d$. 

Let $H_{k\ell}^\lambda$ be the scaling matrix such that
$$
H_{k\ell}^\lambda\vec{e_{i}}=\vec{e_{i}}\quad ({\ccr i\neq k,l}),\quad
{H_{k\ell}\vec{e_k}=\lambda\vec{e_k}},\quad
H_{k\ell}^\lambda\vec{e_\ell}=\lambda^{-1}\vec{e_\ell}.
$$
We show any $H_{ij}^\lambda, $ with $\ccr \lambda\equiv 1(\text{mod } p^{4(L-1)})$ can be written as a product of (conjugates of) matrices produced from Step 2 and Step 3.  This simply follows from that any 2 by 2 matrix { of determinant 1 }that is congruent to $I(\text{mod } p^{4(L-1)})$ can be written as a product 
$$
\mat{1&x\\0&1}\cdot \mat{1&0\\y&1}\cdot \mat{1&z\\0&1}\cdot \mat{1&0\\w&1}.
$$
with $p^{2(L-1)}|x,y,z,w$. 

\medskip

{\bf Step 5.} We can finally prove the proposition for the case $q=p^r$.  For a given $\gamma\in G_{p^r}(p^{4(L-1)})$, we work in reverse order.  We multiply $\gamma$ on the left by matrices produced from previous steps to reach the identity matrix.   

\medskip

We first left multiply $\gamma$ by some unipotent matrix $C_1$ of the form 
$$
C_1=\mat{1&0&\cdots &0\\ z_2&1&\cdots&0\\\vdots&\vdots&\ddots&\vdots \\ z_{d}&0&\cdots&1}, \quad p^{2(L-1)}|z_2,z_3,\hdots,z_d,
$$
which can be produced from Step 2, 
so that $(C_1 \gamma)_{j,1}=0$ for $2\leq j\leq d$.  Similarly, We left multiply $C_1\gamma$ by some matrix $C_2$, which is of the form 
$$
C_2=\mat{1&0&\cdots &0\\ 0&1&\cdots&0\\\vdots&\vdots&\ddots&\vdots \\ 0&z_d'&\cdots&1}, \quad  p^{2(L-1)}| z_3',\hdots,z_d', 
$$
and can be produced from Step 2, so that
$$
(C_2C_1 \gamma)_{2,j}\equiv 0\ (\text{mod } p^r),\quad {\ccr 3\leq j\leq n}.
$$
Iterating this for another $d-2$ steps, we can find $C_3,\,\hdots,\, C_d$, such that 
$$
C_dC_{d-1}\cdots C_1\gamma=M,
$$
where $M$ is upper triangular.  We then multiply by $d-1$ scaling matrices $H_1,\,\hdots,\,H_{d-1}$ produced from Step 4, so that $H_d\cdots H_1 M$ is unipotent and upper-triangular; this matrix can be produced from Step 3.  We can then multiply $H_d\cdots H_1 M$ by the inverse of this conjugate to retrieve the identity.  \par
\medskip

{{\bf Step 6.} Now we deduce the proposition for the general case $q=\prod_{i=1}^{r_0}p_i^{\alpha_i}$ from the special case $q=p^r$.  Let $g\in G_{q}(q_2)$.  For each $p_i^{\alpha_i}||q$, we have $g(\text{mod }p_i^{\alpha_i})\in G_{p_i^{\alpha_i}}^{p^{4(L-1)}}$.  Applying the conclusion of the proposition for prime power moduli,  there exist $h_{p_i^{\alpha_i},1},h_{p_i^{\alpha_i},2},\cdots, h_{p_i^{\alpha_i},N(d)}\in \{g_0^{\pm}(\text{mod }p_i^{\alpha_i}), g_0'^{\pm}(\text{mod }p_i^{\alpha_i})\} $, and $g_{p_i^{\alpha_i},1},\cdots, g_{p_i^{\alpha_i},N(d)}\in G_{p_i^{\alpha_i}}(p^{L-1})$ such that

\begin{align}\label{0216}
g_{p_i^{\alpha_i},1}\cdot h_{p_i^{\alpha_i},1}\cdot g_{p_i^{\alpha_i},1}^{-1}\cdots g_{p_i^{\alpha_i},N(d)}\cdot h_{p_i^{\alpha_i},N(d)}\cdot g_{p_i^{\alpha_i},N(d)}^{-1}\equiv g(\text{mod }p_i^{\alpha_i})
\end{align}

Moreover, the proof in Step 1 - Step 5 actually gaurantees that for each $1\leq j\leq N(d)$, we can choose these $h_{p_i^{\alpha_i},j}$ that come from reduction of an common element $h_j\in \{g_0^{\pm}, g_0'^{\pm}\}$, that is, $h_{p_i^{\alpha_i},j}=h_j(\text{mod }p_i^{\alpha_i})$ for each $ 1\leq i\leq r_0$. 
 
The Strong-Approximation Property for $\text{SL}_d$ (\cite{Nori1987}) says that 
$$\text{SL}_d(\mathbb Z/q\mathbb Z)\cong \prod_{i=1}^{r_0}\text{SL}_d(\mathbb Z/p_i^{\alpha_i}\mathbb Z) $$

Under this identification,  $G_{q}(q_2)\cong  \prod_{i=1}^{r_0}G_{p_i^{\alpha_i}}({p^{4(L-1)}})$.  This shows for each $1\leq j\leq N(d)$, we can find $g_j\in G_q(q_2)$, such that $g_j\equiv g_{p_i^{\alpha_i},j}(\text{mod }p_{i}^{\alpha_i})$ for any $1\leq i\leq r_0$. Then it follows from \eqref{0216} that 
$$g_{1}\cdot h_{1}\cdot g_{1}^{-1}\cdots g_{N(d)}\cdot h_{N(d)}\cdot g_{N(d)}^{-1}\equiv g( \text{mod }q )$$
 }

\end{proof}

\subsection{Proof of the main theorem} In this section we take $G=\mathbb G(\mathbb Z)$. Since $\text{Zcl}(\Gamma)=\text{SL}_{d_1}\times\text{SL}_{d_2}\times\cdots\times \text{SL}_{d_k}$ or $\text{SA}_d$, the Strong Approximation Property for $\Gamma$ (\cite{Nori1987}) says that under the inclusion 
$$ i: \Gamma\hookrightarrow \mathbb G(\hat{\mathbb Z}), $$
the closure $\overline{i(\Gamma)}$ of $i(\Gamma)$ is an open and thus cofinite subgroup of $\mathbb G(\hat{\mathbb Z})$.  
{ This implies that there exists an integer $\mathcal B=\prod_{i=1}^{n_{\mathcal B}}{p_i^{\beta_i}}$ such that for any $\mathcal B|q$, we have 
$$\Gamma(\mathcal B)/\Gamma(q)= G(\mathcal B)/G(q).$$
We let $L_0\geq 2$ be the smallest integer such that for each prime $p$, $\mathbb Z/p^{L_0}\mathbb Z$ has at least $d$ distinct elements.  Then we take 
\begin{align}\label{1503}
L=\max\{L_0, \beta_i+1: 1\leq i\leq n_{\mathcal B}\}. 
\end{align}
}
We let $\text{d}_{S, q}(\cdot,\cdot)$ denote the distance function on $X_q$.  We only need to consider $q\in \mathbb N_{L}:= \{q\in\mathbb N: q_0^{5L-5}|q\}$: Suppose for the set $\mathbb N_{L}$, the inequality \eqref{inequality1} holds for some constant $c_L$. Then for a general $q$, we can lift any two points $p_1,p_2\in X_q$ to two points $\tilde p_1, \tilde p_2$ on $X_{q^{5L-5}}$.  By assumption, $\tilde p_1$ and $\tilde p_2$ is connected by a path with length $\leq c_L\log {q^{5L-5}}= c_L(5L-5) \log q$. This path projects down to a path in $X_q$ connecting $p_1$ and $p_2$ and of length $\leq c_L(5L-5) \log q$.  Therefore, $\text{d}_{S,q}(p_1,p_2)\leq c_L(5L-5)\log q$.   \par
{
Our next reduction is that, given $q\in N_L$,  it suffices to prove for any $p_1\equiv I(q_0^{5L-5}) (\text{mod }q)$, 
$$\text{d}_{S, q}(I, p_1)\leq c_L'\log q $$
for some $c_L'>0$.  This is because by Corollary \ref{1611}, Case (2), for any $p_1\in X_q$, one can find $p_1'\in X_q\cap G_q(q_0^{5L-5})$, such that  $$d_{S,q}(p_1,p_1')\leq c_L''\log q$$ for some $c_L''>0$.  
}

 We denote $\mathbb N_L^{*}:=\{\prod_{i=1}^k{p_i^{5L-5}}: p_i\text{'s are distinct primes}\}$.\par

\medskip 

\noindent {\bf Case 1:} $\normalfont\text{Zcl}(\Gamma)=\text{SA}_d$. \par

Let $c_{1}$ be the implied constant by Corollary \ref{1611}, Case (1), for the group $\theta(\Gamma)$ with the generating set $\theta(S)$, and let $c_{2,L}$ be the implied constant by Theorem \ref{GV2012-2019} for $\Gamma$ with moduli restricted to the set $\mathbb N_{L}^{*}$.


In the Cayley graph $X_q$, {\ccr Corollary \ref{1611}, Case (2)} allows us to choose three vertices $(T_1,v_1), (T_1', v_1')$ and $(T_2,v_2)$ satisfying the following conditions

\begin{equation*}
\begin{cases}
&T_1 \text{ satisfies the congruence condition of }g_0 \text{ in Proposition \ref{key-proposition}},\\
& v_1\equiv 0 (\text{mod } q_0^{5(L-1)}),\\
& \text{d}_{S,q}((I,\vec{0}), (T_1,v_1))\leq c_{2,L}\log q,
\medskip\\
&{\ccr T_1' \text{ satisfies the congruence condition of }g_0' \text{ in Proposition \ref{key-proposition}}},\\
&{\ccr v_1'\equiv 0 (\text{mod } q_0^{5(L-1)}),}\\
&{\ccr \text{d}_{S,q}((I,\vec{0}), (T_1',v_1'))\leq c_{2,L}\log q, }
\medskip \\
&T_2\equiv I(\text{mod } q_0^{4L-4})\\
& v_2\equiv (q_0^{L-1},0,\cdots,0)^t (\text{mod } q_0^L),\\
& \text{d}_{S,q}((I,\vec{0}), (T_2,v_2))\leq c_{2,L}\log q.
\end{cases}
\end{equation*}



By Proposition \ref{key-proposition}, there exists $S_1,\cdots, S_{N(d)}\in G_q(q_0^{2(L-1)})$, $(\mathfrak T_i, \mathfrak v_i)=(T_1,v_1)\text{ or }(T_1', v_1')$, and $\epsilon_i=\pm 1, 1\leq i\leq N(d)$, such that
$$\prod_{1\leq i\leq N(d)}S_i \mathfrak T_i^{\epsilon_i} S_i^{-1}=T_{2}^{-1}$$ 
Corollary \ref{1611}, Case (1) implies for each $1\leq i\leq N(d)$ we can find $w_i\in (\mathbb Z/q\mathbb Z)^d$, such that 
$$\text{d}_{S,q}((I,\vec{0}),(S_i, w_i))\leq c_1 \log q.$$

Then we have $$\left(\prod_{1\leq i\leq N(d)}(S_i,w_i)(\mathfrak T_i, \mathfrak v_i)^{\epsilon_i}(S_i,w_i)^{-1} \right)\cdot (T_2,v_2)=(I, v_0),$$
 for some {\ccr{$v_0\equiv (p^{L-1},0,\cdots,0) (\text{mod } p^L)$ }} for any $p|q$.  
 Therefore, we have 
 $$\text{d}_{S,q}((I,\vec{0}), (I, v_0)))\leq (N(d)(2c_1+c_{2,L})+c_{2,L})\log q.$$
 
The following identity, although elementary, is crucial in our argument for this case: 
\begin{align}\label{key}(T,v)^{-1}\cdot(I,v_0)\cdot (T,v)=(I, T^{-1}v_0).
\end{align}

%
It is fairly straightforward to show that 
\begin{align*}
G_{q}(q_0^{4(L-1)})\cdot v_0+G_{q}(q_0^{4(L-1)})\cdot v_0\supset \{v\in \mathbb Z/q_0^r\mathbb Z: v\equiv (0,\cdots,0) ({\text{mod }}p^{5(L-1)})\}.
\end{align*}
Then for any $v\in (\mathbb Z/q\mathbb Z)^d$ with $v\equiv (0,\cdots,0)({\text{mod }}q_0^{5(L-1)})$, we have 

\begin{align}\label{1534}
\text{d}_{S,q}((I,\vec{0}), (I, v)))\leq  ((4+4N(d))c_1+(2+2N(d))c_{2,L})\log q.  
\end{align}

Now for any $(T,v)\in \text{SL}_d(\mathbb Z/q\mathbb Z)\ltimes (\mathbb Z/q\mathbb Z)^d$,  applying Proposition \ref{key-proposition} again we can find $(S_i', v_i'), 1\leq i\leq N(d)$ with $d_{S,q}((I,\vec{0}), (S_i,w_i))\leq c_1\log q$, such that 
$$\left(\prod_{1\leq i\leq N(d)}(S_i',w_i')(\mathfrak T_i', \mathfrak v_i')^{\epsilon_i}(S_i,w_i)^{-1} \right)\cdot (T,v)=(I, v'),$$
with $v'\equiv(0,\cdots,0)({\text{mod }}p_0^{5(L-1)})$.  \eqref{1534} then implies
 $$\text{d}_{S,q}((I,\vec{0}),(T,v))\leq ((4+6N(d))c_1+(2+3N(d))c_{2,L})\log q.$$

\medskip

{\bf{Case 2:}} {$\text{Zcl}(\Lambda)=\text{SL}_{d_1}\times\text{SL}_{d_2}\times \cdots\times \text{SL}_{d_k} $}.\par
We only work with the case $k=2$ and $d_1=d_2=d$.  The idea for the general case is the same.  Let $\text{pr}_1$ ($\text{pr}_2$, respectively) be the projection from $\text{SL}_d\times \text{SL}_d$ to the first (second, respectively) factor.  Let $c_{3}$ be the implied constant by Corollary \ref{1611}, Case (1) for $\text{pr}_1(\Gamma)$ with the generating set $\text{pr}_1(S)$, $c_{4}$ be the implied constant by Corollary \ref{1611}, Case (1) for $\text{pr}_2(\Gamma)$ with the generating set $\text{pr}_2(S)$, and let $c_{5,L}$ be the implied constant by Corollary \ref{1611}, Case (2) for $\Gamma$ with moluli restricted to the set $\mathbb N_L^{*}$.  We want to show that a general element $(P,Q)\in G_{q}(q_0^{5(L-1)})$ has logarithmic distance to $(I,I)$. \par

Applying Corollary \ref{1611}, Case (2), we can find {$(P_1,Q_1),(P_1', Q_1'), (P_2, Q_2), (P_2', Q_2')$}, such that 

\begin{equation*}
\begin{cases}
&P_1, Q_2 \text{ satisfy the congruence conditions for } g_0 \text{ in Proposition \ref{key-proposition}},\\
&{P_1', Q_2' \text{ satisfy the congruence conditions for } g_0' \text{ in Proposition \ref{key-proposition}},}\\
& Q_1, Q_1', P_2, P_2' \equiv I(\text{mod } q_0^{5L-5}),\\
& \text{d}_{S,q}((I,I),(P_1,Q_1))\leq c_{5,L}\log q,\\
& \text{d}_{S,q}((I,I),(P_1',Q_1'))\leq c_{5,L}\log q,\\
&  \text{d}_{S, q}((I,I),(P_2,Q_2))\leq c_{5,L}\log q,\\
&  \text{d}_{S, q}((I,I),(P_2',Q_2'))\leq c_{5,L}\log q.
\end{cases}
\end{equation*}

{Next, we apply Proposition \ref{key-proposition} to $g_0= P_1$ and $g_0'=P_1'$.  We can find $\tilde{P_i}\in \text{pr}_1(G)(q_0^{L-1}), P_{(i)}\in\{P_1^{\pm 1}, P_1'^{\pm 1}\}$ for $1\leq i\leq N(d)$, such that 
\begin{align}\label{0021}
\prod_{1\leq i\leq N(d)}\tilde{P_i} P_{(i)} \tilde{P_i}^{-1} = P_2^{-1}
\end{align}
Similarly, we can find $\tilde{P_i}^*\in\text{pr}_1(G)(q_0^{L-1}), P_{(i)}^*\in\{P_1^{\pm 1}, P_1'^{\pm 1}\}$ for $1\leq i\leq N(d)$, such that 
\begin{align}\label{0022}
\prod_{1\leq i\leq N(d)}\tilde{P_i}^* P_{(i)}^* {\tilde{P_i}}^{*-1} = P_2'^{-1}
\end{align}

For each $P_{(i)},  P_{(i)}^*$, we assign $Q_{(i)}, Q_{(i)}^*$ so that $$(P_{(i)}, Q_{(i)}), (P_{(i)}^*, Q_{(i)}^*) \in \{(P_1, Q_1), (P_1', Q_1'), (P_1^{-1}, Q_1^{-1}), (P_1'^{-1}, Q_1'^{-1}) \} $$

As in the previous case, using Corollary \ref{1611}, Case (1), we can find $(\tilde{P_i},\tilde{Q_i}), (\tilde{P_i}^*,\tilde{Q_i}^*)\in G_q(q_0^{L-1})$, such that 
$$\text{d}_{S,q}((I,I),(\tilde{P_i},\tilde{Q_i}))\leq c_3 \log q.$$

Then \eqref{0021} and \eqref{0022} give 
$$\left(\prod_{1\leq i\leq N(d)}(\tilde{P_i},\tilde{Q_i})(P_{(i)},Q_{(i)})(\tilde{P_i}^{-1}, \tilde{Q_i}^{-1})\right)\cdot (P_2,Q_2)=(I,Q_3),$$
and 
$$\left(\prod_{1\leq i\leq N(d)}(\tilde{P_i}^*,\tilde{Q_i}^*)(P_{(i)}^*,Q_{(i)}^*)(\tilde{P_i}^{*-1}, \tilde{Q_i}^{*-1})\right)\cdot (P_2',Q_2')=(I,Q_3'),$$
for some $Q_3$ and $Q_3'$.  The matrices $Q_3$, $Q_3'$ satisfy the congruence conditions for $g_0$, $g_0'$ in Proposition \ref{key-proposition}, respectively.  This is because $Q_1, Q_1'\equiv I(q_0^{5(L-1)})$, and  $Q_2, Q_2'$ satisfy the congruence conditions for $g_0, g_0'$ in Proposition \ref{key-proposition}, respectively. \par
We thus have 
\begin{equation}\label{1429}
\text{d}_{S,q}((I,I),(I,Q_3))\leq N(d)(2c_3+c_{5,L})\log q,
\end{equation}
and
\begin{equation}\label{1430}
\text{d}_{S,q}((I,I),(I,Q_3'))\leq N(d)(2c_3+c_{5,L})\log q,
\end{equation}

Applying Proposition \ref{key-proposition} to $g_0=Q_3, g_0'=Q_3'$, since $Q\equiv I(\text{mod } q_0^{5L-5})$, we can find $Q_i^{\star}\in \text{pr}_2(G)(q_0^{L-1})$, and a choice $Q_{(i)}\in\{Q_3^{\pm}, Q_3'^{\pm}\}$, such that
$$\prod_{i=1}^{N(d)} Q_i^\star Q_{(i)} { Q_i^\star}^{-1}=Q.$$

Furthermore, applying Corollary \ref{1611}, Case (2) to $\text{pr}_2(\Gamma)$, we can find ${{P_i}^{\star}}\in \text{pr}_1(G)(q_0^{L-1})$, $1\leq i\leq N(d)$, so that 
\begin{equation}\label{1431}
\text{d}_{S, q}((I,I),({{P_i}^\star},{{Q_i}^\star}))\leq c_4\log q.
\end{equation}
We thus have 
\begin{equation}\label{1433}\ccr \prod_{i=1}^{N(d)} ({{P_i}^\star}, {{Q_i}^\star})(I, Q_{(i)})({{P_i}^\star}, {{Q_i}^\star})^{-1}=(I, Q ).
 \end{equation}
Collecting \eqref{1429}, \eqref{1431}, \eqref{1433}, we obtain 
 $$\text{d}_{S,q}((I,I), (I,Q))\leq (2N(d)^2c_3+2N(d)c_4+N(d)^2c_{5,L})\log q.$$
 }
Similarly,

 $$\text{d}_{S,q}((I,I), (P, I))\leq (2N(d)^2c_4+2N(d)c_3+N(d)^2c_{5,L})\log q.$$

Finally, since $(P,Q)=(P, I)(I, Q)$, we have 

\begin{align*} &\text{d}_{S,q}((I,I),(P,Q))\\&\ccr \leq ((2N(d)^2+2N(d))c_3+(2N(d)^2+2N(d))c_4+2N(d)^2c_{5,L})\log q.\end{align*}

\bibliographystyle{plain}
\bibliography{logbound}

\end{document}